\documentclass{amsart}
\usepackage{mathtext}
\usepackage{amsmath}
\usepackage{amsfonts}
\usepackage{amssymb}
\usepackage{mathrsfs}
\usepackage{amsthm}
\usepackage{dsfont}
 
\theoremstyle{definition}\newtheorem{Def}{Definition}
\theoremstyle{plain}\newtheorem{Th}{Theorem}
\theoremstyle{remark}\newtheorem{Rem}{Fact}
\theoremstyle{plain}\newtheorem{Le}{Lemma}
\theoremstyle{plain}

\newcommand{\Cii}[1]{_{{}_{\scriptstyle #1}}}

\newcommand{\df}{\buildrel\mathrm{def}\over=}

\DeclareMathOperator{\supp}{supp}
\newcommand{\dltn}{\Delta}
\DeclareMathOperator{\sign}{sign}
\newcommand{\mean}{\mathbb{E}}

\author{Nikolay N. Osipov}
\title[Rubio de Francia inequality for the Walsh system]{Littlewood--Paley--Rubio de Francia inequality for the Walsh system}
\thanks{This work was carried out during the tenure of an ERCIM ``Alain Bensoussan'' Fellowship Programme}

\thanks{The author is also supported by RFBR (grant no. 14-01-31163 and no. 14-01-00198).}
\address{St.~Petersburg Department of Steklov Mathematical Institute RAS, Fontanka 27, St.~Petersburg, Russia}
\email{nicknick AT pdmi DOT ras DOT ru}

\address{Norwegian University of Science and Technology (NTNU), IME Faculty, Dep. of Math. Sci., Alfred Getz' vei 1, Trondheim, Norway}
\email{nikolai DOT osipov AT math DOT ntnu DOT no}


\begin{document}
\begin{abstract}
	Rubio de Francia proved the one-sided Littlewood--Paley inequality for arbitrary intervals in $L^p$, $2\le p<\infty$.
	In this article, such an inequality is proved for the Walsh system.
\end{abstract}

\maketitle
\setcounter{secnumdepth}{1}
\section{Formulation of the result}
First, we make some agreement about our notation. From now on, by the space~$L^p$ we mean the space $L^p([0,1])$. 
Also, by $L^p(l^2)$ we mean 
$L^p([0,1], l^2)$ (i.e., the space of $l^2$-valued functions on the interval $[0,1]$).

Let $I_m$ be mutually disjoint intervals in $\mathbb{Z}$ (here and below, we assume that $m$ runs over some finite or countable set). 
In 1983, Rubio de Francia proved (see~\cite{Ru}) that
\begin{equation}\label{RFEstim}
	\Big\|\Big(\sum_m \big|(\widehat{f}\,\,\mathds{1}_{I_m})^{\vee}\big|^2\Big)^{1/2}\Big\|_{L^p} \le C_p \|f\|\Cii{L^p},\quad 2 \le p \le \infty,
\end{equation}
where the constant $C_p$ does not depend on the intervals~$I_m$ or the function~$f$. It is worth noting that he considered the whole line $\mathbb{R}$ rather than the interval $[0,1]$ (so $I_m$ were intervals in $\mathbb{R}$, not in $\mathbb{Z}$), but this fact did not play a significant role in his considerations.

By duality, estimate~\eqref{RFEstim} is equivalent to the following:
\begin{equation}\label{DualEstim}
	\Big\|\sum_{m}f_{m}\Big\|_{L^p} \le
	C_p\big\|\{f_m\}\big\|_{L^p(l^2)},\quad 1<p \le 2,
\end{equation}
where $\{f_m\}$ is a sequence of functions such that $\supp{\widehat{f}_{m}} \subset I_{m}$. In fact, it is already known that 
estimate~\eqref{DualEstim} remains true for $p \in (0, 1]$ (see~\cite{Bo} for $p=1$ and~\cite{KiPa} for all $p\in(0,1]$).

Our goal is to prove an analogue of~\eqref{DualEstim} for the situation where we use the Walsh system instead of the exponential functions.
We give the corresponding definition.
\begin{Def}\label{DefOfWalsh}
\emph{The Walsh system} $\{w_n\}\Cii{n\in\mathbb{Z}_+}$ consists of step functions on the interval $[0,1]$ that are defined as follows. 
First, we set $w_0 \equiv 1$. Next, for any index $n>0$ we consider its dyadic decomposition $n = 2^{k_1} + \cdots + 2^{k_s}$, 
$k_1 >k_2>\cdots> k_s \ge 0$, and set
$$
	w_n(x) \df \prod_{i=1}^s r_{k_i+1}(x),
$$
where $r_k$ are the Rademacher functions, that is $r_k(x) = \sign\sin 2^k\pi x$.
\end{Def}
The Walsh functions form an orthonormal basis in $L^2$ (see, e.g., \cite[IV.5]{KaSa}). In the next section, we will discuss their properties in more detail. Now we present  
the corresponding analogue of Rubio de Francia's result.
\begin{Th}\label{th1}
	Let $I_m$ be mutually disjoint intervals in~$\mathbb{Z}_+$. Let $f_m$ be functions such that
	$$
		f_m = \sum_{n\in I_m} (f_m, w_n)\, w_n.
	$$
	If $1<p\le 2$\textup, then
	$$
		\Big\|\sum_m f_m\Big\|_{L^{p}} \le C_p \big\|\{f_m\}\big\|_{L^p(l^2)},
	$$
	where $C_p$ does not depend on the collections $\{I_m\}$ and~$\{f_m\}$.
\end{Th}
The proof of this theorem will be close in spirit to arguments in~\cite{Ru} or~\cite{KiPa}. However there will be some interesting combinatorial considerations that do not occur in the case of the trigonometric basis. On the other hand, some parts of our proof will be much easier due to the discrete nature of the Walsh system.

\section{Preliminaries}
\subsection{Concerning the Walsh system} 
Here we define a certain group operation on~$\mathbb{Z}_+$ and describe its connection with the Walsh functions. 
\begin{Def}\label{DefOfPlus}
Let $a$ and $b$ be numbers in $\mathbb{Z}_+$. Consider their dyadic decompositions
$$
	a = \sum_{k = 0}^{\infty} \theta_k(a)\, 2^{k} \quad\mbox{and}\quad  b = \sum_{k=0}^{\infty} \theta_k(b)\, 2^{k},
$$
where the functions~$\theta_k$ can take the values $0$ or $1$.
Then, we set
$$
	a \dotplus b \,\df\, \sum_{k=0}^{\infty}  \big((\theta_k(a) + \theta_k(b)) \bmod 2\big)\, 2^{k}.
$$
\end{Def}
\begin{Rem}\label{ZIsGroup}
	The set~$\mathbb{Z}_+$, together with the operation $\dotplus$, is an abelian group whose elements are inverse to themselves: 
	$a \dotplus a = 0$, $a \in \mathbb{Z}_+$.
\end{Rem}
\begin{Rem}\label{WIsGroup}
	The Walsh system is an abelian group with respect to multiplication that is isomorphic to the group $\mathbb{Z}_+$ with 
	operation~$\dotplus$. Namely, we have 
	$$
		w_a(x)\, w_b(x) = w_{a\dotplus b}(x)
	$$
	for a.e. $x \in [0,1]$ and for any $a,b \in \mathbb{Z}_+$.
\end{Rem}
This two facts follow directly from Definitions~\ref{DefOfWalsh} and~\ref{DefOfPlus}. A more detailed discussion of the Walsh functions and the operation~$\dotplus$ can be found, for example, \mbox{in~\cite[IV.5]{KaSa}}.

\subsection{Dyadic martingales}
	Let $\mathcal{F}_0 \subset \mathcal{F}_1 \subset \dots$ be the increasing sequence of $\sigma$-algebras on $[0,1]$ where each 
	$\mathcal{F}_k$ is generated by the dyadic subintervals of length~$2^{-k}$.
	We introduce the following notation:
	$$
		\mean_k f\df \mean[f|\mathcal{F}_k]  = \sum_{i=0}^{2^k - 1}  \frac{\mathds{1}_{A_i}}{|A_i|} \int\limits_{A_i} f(x)\,dx,
	$$
	where $f$ is a function in $L^1$ and $A_i = [i\,2^{-k}, (i+1)\,2^{-k}]$. 
\begin{Def}
	Consider a collection $\mathcal{M} = \{M_k\}\Cii{k \in \mathbb{Z}_+}$ of scalar-valued or $l^2$-va\-lu\-ed integrable functions on $[0,1]$.
	We say that $\mathcal{M}$ is \emph{a dyadic martingale} (scalar-valued or $l^2$-va\-lu\-ed, respectively) if each $M_k$ is 
	$\mathcal{F}_k$-measurable and 
	$\mean_{k} M_{k+1} = M_{k}$.
\end{Def}
From now on, by martingales we mean dyadic martingales. 
The general concept of vector-valued martingales (not only dyadic ones) is described in detail, for example, in~\cite{DiUh}.
We will use the following notation:
$$
	\dltn_0\mathcal{M} \df M_0\quad\mbox{and}\quad\dltn_k\mathcal{M} \df M_k-M_{k-1},\; k>0.
$$

The $L^p$-norms for martingales are defined as follows:
$$
	\|\mathcal{M}\|\Cii{L^p} \df \sup_{k} \|M_k\|\Cii{L^p}, \quad 1 \le p < \infty.
$$
If $1<p<\infty$, then each martingale $\mathcal{M} = \{M_k\}$ with $\|\mathcal{M}\|_{L^p} \le \infty$
can be identified with some function $f \in L^p$ and vice versa: the functions $M_k$ have a limit $f$ in $L^p$ such that 
$\|\mathcal{M}\|_{L^p} = \|f\|_{L^p}$ (see, e.g., \cite[V.2]{DiUh}), and, on the other hand, if $f$ is a function in $L^p$, then the sequence 
$\{\mean_k f\}\Cii{k \in \mathbb{Z}_+}$ is a martingale with the same norm. As for the case $p=1$, the condition 
$\|\mathcal{M}\|_{L^1} \le \infty$ is not sufficient for the existence of the $L^1$-limit, but each function $f \in L^1$ can still be treated as 
the martingale $\mathcal{M} = \{\mean_k f\}$ and we have $\|\mathcal{M}\|_{L^1} = \|f\|_{L^1}$.

The above considerations justifies the following notation: for $f \in L^1$ we set
$$
	\dltn_0 f \df \mean_0 f\quad\mbox{and}\quad\dltn_k f \df \mean_k f-\mean_{k-1} f,\; k>0.
$$
We also introduce a collection of dyadic intervals in $\mathbb{Z}_+$:
$$
	\delta_0 \df \{0\}\quad\mbox{and}\quad \delta_k \df [2^{k-1}, 2^{k} - 1],\;k>0.
$$
The following fact shows the connection between dyadic martingales and the Walsh system.
\begin{Rem}\label{MartViaWalsh} 
For any $f \in L^1$, we have
$$
	\mean_k f = \sum_{n = 0}^{2^k - 1} (f, w_n)\, w_n \quad\mbox{and}\quad \dltn_k f = \sum_{n\in \delta_k} (f, w_n)\, w_n.
$$
\end{Rem}
This simple and well-known fact follows, for example, from arguments in \cite[IV.5]{KaSa}.

\subsection{Operators on martingales} Consider the space of \emph{simple} martingales (we say that a martingale~$\mathcal{M}= \{M_k\}$ is simple if $M_k = M_{k+1}$ for all sufficiently large~$k$).  
We suppose it consists of martingales that are either all scalar-valued or all $l^2$-va\-lu\-ed.
Let~$T$ be an operator (not necessarily linear) that is defined on this space and transforms martingales into scalar-valued measurable functions. Suppose it satisfies the following conditions:
\begin{enumerate}
	\item[(a)] $|T(\mathcal{M}_1+\mathcal{M}_2)| \le C_1(|T\mathcal{M}_1|+|T\mathcal{M}_2|)$;
	\item[(b)] $\|T\mathcal{M}\|_{L^2} \le C_2 \|\mathcal{M}\|_{L^2}$;
	\item[(c)] if a martingale $\mathcal{M} = \{M_k\}\Cii{k \in \mathbb{Z}_+}$ satisfies the relations $M_0 \equiv 0$ and 
	$$\dltn_k\mathcal{M} = \mathds{1}_{e_k}\dltn_k\mathcal{M}\quad\mbox{for}\quad k>0,$$ where $e_k \in \mathcal{F}_{k-1}$, then
	$$\big\{|T\mathcal{M}|>0\big\} \subset \bigcup_{k = 1}^{\infty} e_k.$$
\end{enumerate}
For such an operator we can state the following theorem (it was proved for sca\-lar-va\-lu\-ed martingales in~\cite{Gu} and was modified for vector-valued martingales in~\cite{Ki}).
\begin{Th}\label{OperOnMart}
	If an operator~$T$ satisfies conditions~\textup{(a), (b),} and~\textup{(c),} then for simple martingales~$\mathcal{M}$ we have the weak type $(1,1)$ estimate\textup:
	$$
		\big|\big\{|T\mathcal{M}|>\lambda\big\}\big| \le \mathrm{const}\,{\lambda}^{-1}{\|\mathcal{M}\|\Cii{L^1}} \quad\mbox{for}\quad \lambda > 0,
	$$
	where the constant depends only on $C_1$ and $C_2$.
\end{Th}
Note that it is presented in greater generality in~\cite{Ki}: martingales are $X$-valued (where $X$ is an arbitrary Banach space), they are not supposed to be dyadic, and a weaker condition is imposed instead of condition~(b). 

\section{Auxiliary lemmas}
Here we prove some auxiliary propositions. We start with a lemma that describes how the operation $\dotplus$ transforms intervals in~$\mathbb{Z}_+$.
\begin{Le}\label{intervals}
Let $N$ be some number in $\mathbb{Z}_+$. Consider its dyadic decomposition\textup: $N = 2^{k_1} + \cdots + 2^{k_s}$\textup, where 
$k_1 >k_2>\cdots> k_s \ge 0$. Also we introduce the collection
$$
	\{\varkappa_j\}_{j=1}^{\infty} \df \mathbb{Z}_+ \setminus \{k_i\}_{i=1}^{s}
$$
ordered by ascending\textup: $\varkappa_1 < \varkappa_2 < \cdots < \varkappa_j < \cdots$.
Then
\begin{equation}\label{ZTransf}
	[0,N-1]\dotplus N = \bigcup_{i =1}^s \delta_{k_i+1}\quad\mbox{and}\quad 
	[N, +\infty) \dotplus N = \delta_0\cup\Big(\bigcup_{j =1}^\infty \delta_{\varkappa_j+1}\Big).
\end{equation}
More precisely\textup, we have
\begin{equation}\label{shifts}
\begin{aligned}[]	
	[0,\, 2^{k_1}-1] & \dotplus N = \delta_{k_1+1}; \\
	[2^{k_1},\, 2^{k_1}+2^{k_2}-1] & \dotplus N = \delta_{k_2+1}; \\
	&\hskip4.5pt\vdots \\
	\big[\,{\textstyle\sum_{l=1}^{s-1} 2^{k_l}},\, N -1\big] & \dotplus N = \delta_{k_s+1};\\
	\{N\} & \dotplus N = \delta_0; \\
	[N+1,\, N + 2^{\varkappa_1}] & \dotplus N = \delta_{\varkappa_1+1};\\
	&\hskip4.5pt\vdots \\
	\big[N +  {\textstyle \sum_{l=1}^{j-1} 2^{\varkappa_l}} + 1,\, N + {\textstyle \sum_{l=1}^{j} 2^{\varkappa_l}}\big] & \dotplus N = \delta_{\varkappa_j+1};\\
	&\hskip4.5pt\vdots
\end{aligned}
\end{equation}
\end{Le}
\begin{proof}
	It is worth noting that the first of identities~\eqref{ZTransf} and its proof can be found in \cite[IV.5]{KaSa}. The corresponding identities for the intervals 
	(the first $s$ identities in~\eqref{shifts}) can also be derived from that proof. 
	
	Here we provide a complete proof of the lemma.	
	Consider the set
	\begin{equation*}
		Q_i \df\big[{\textstyle\sum_{l=1}^{i - 1}2^{k_l}},\,{\textstyle\sum_{l=1}^{i}2^{k_l}} - 1 \big]\dotplus N,\qquad 1\le i \le s.
	\end{equation*}
	We denote $$\sigma \df \sum_{l=1}^{i-1}2^{k_l}\quad\mbox{and}\quad \gamma \df \!\sum_{l=i+1}^{s}\!2^{k_l}.$$
	By Definition~\ref{DefOfPlus} and Fact~\ref{ZIsGroup}, we have
	$$
		Q_i = \big\{(\sigma \dotplus v)\dotplus (\sigma\dotplus 2^{k_i}\dotplus\gamma)\big\}_{v = 0}^{2^{k_i} -1} = [0,\, 2^{k_{i}} - 1] \dotplus 2^{k_i}\dotplus \gamma.
	$$
	Fact~\ref{ZIsGroup} implies that the set $[0,\, 2^{k_{i}} - 1] \dotplus \gamma$ consists of $2^{k_{i}}$ numbers. On the other hand, all 
	these numbers are lesser than $2^{k_{i}}$ because $\gamma < 2^{k_{i}}$ (see Definition~\ref{DefOfPlus} again). Thus, we obtain $[0,\, 2^{k_{i}} - 1] \dotplus \gamma = [0,\, 2^{k_{i}} - 1]$. Finally, we have
	$$
		Q_i = [2^{k_i},\, 2^{k_i + 1} - 1] = \delta_{k_i+1}.
	$$
	
	Next, we consider the set
	$$
		U_{j} \df \big[N +  {\textstyle \sum_{l=1}^{j-1} 2^{\varkappa_l}} + 1,\, N + {\textstyle \sum_{l=1}^{j} 2^{\varkappa_l}}\big]  \dotplus N,\qquad j \ge 1.
	$$
	We denote
	$$
		\mu \, \df \!\!\sum_{l \colon \! k_l > \varkappa_j} \!\! 2^{k_l} \quad\mbox{and}\quad \eta \, \df \!\!\sum_{l \colon \! k_l < \varkappa_j} \!\! 2^{k_l}.
	$$
	By the definition of the sequence~$\{\varkappa_j\}$, we have
	$$
		U_{j} = \big[2^{\varkappa_j} +  \mu,\, {\textstyle\sum_{k=0}^{\varkappa_j}2^{k}}+\mu \big]  \dotplus \mu \dotplus \eta
	$$
	We note that 
	$$
		\sum_{k=0}^{\varkappa_j}2^{k} = 2^{\varkappa_j+1}-1
	$$ 
	and that for any integer $v$ such that $2^{\varkappa_j} \le v \le 2^{\varkappa_j+1}-1$, we have $v + \mu = v\dotplus \mu$. Thus, we can see that
	$$
		U_{j} = [2^{\varkappa_j},\, 2^{\varkappa_j+1}-1] \dotplus \eta.
	$$
	This implies that $U_{j}$ consists of $2^{\varkappa_j}$ integers that are not less than $2^{\varkappa_j}$, but are less than $2^{\varkappa_j+1}$.
	Therefore, we have
	\[
		U_{j} = [2^{\varkappa_j},\, 2^{\varkappa_j+1}-1] = \delta_{\varkappa_j+1}.\qedhere
	\]

	%
\end{proof}

Now we consider two auxiliary operators and obtain their $L^p$-boundedness as a consequence of Theorem~\ref{OperOnMart}.
\begin{Le}\label{aboutS} 
	Suppose multi-index $(j,k)$ runs over some subset $\mathcal{A} \subset \mathbb{Z}_+^2$.
	Consider a collection of numbers $\{a_{j,k}\}_{(j,k) \in \mathcal{A}}$ in $\mathbb{Z}_+$ such that 
	$\{a_{j,k} \dotplus \delta_k\}_{(j,k) \in \mathcal{A}}$ is a collection of mutually disjoint subsets in~$\mathbb{Z}_+$.
	Let $h = \{h_{j,k}\}_{(j,k) \in \mathbb{Z}_+^2}$ be a function in $L^p(l^2)$\textup, $1<p\le 2$.
	Suppose the operator~$S$ is defined by the formula
	$$
		Gh \df \!\!\sum_{(j,k) \in \mathcal{A}}\!\! w_{a_{j,k}}  \dltn_k h_{j,k}.
	$$
	Then we have
	\begin{equation}\label{GIsBounded}
		\|Gh\|\Cii{L^p} \le C_p \|h\|\Cii{L^p(l^2)},
	\end{equation}
	where the constant~$C_p$ depends only on~$p$.
\end{Le}
\begin{proof}
	We recall that the Walsh system is an orthonormal basis in $L^2$. Using Parseval's identity together with Facts~\ref{WIsGroup} and~\ref{MartViaWalsh}, 
	we can prove the $L^2$-bo\-un\-ded\-ness of $G$. Indeed, since the sets $a_{j,k} \dotplus \delta_k$ are pairwise disjoint for 
	$(j,k) \in \mathcal{A}$, we have
	$$
		\|Gh\|_{L^2}^2 = \!\!\sum_{(j,k) \in \mathcal{A}}\!\! \|\dltn_k h_{j,k}\|_{L^2}^2 \le \!\!\sum_{(j,k) \in \mathbb{Z}_+^2}\!\! \|h_{j,k}\|_{L^2}^2 = \|h\|_{L^2(l^2)}^2.
	$$
	Since the operator~$G$ is linear and satisfies conditions~(b) and~(c), Theorem~\ref{OperOnMart} implies the weak type (1,1) estimate for $G\mathcal{M}$ if $\mathcal{M}$ is a simple martingale.
	
	Suppose for a while that the set $\mathcal{A}$ is finite. Then, passing to the limit we obtain the weak type (1,1) estimate for $Gh$, where $h$ is any function in $L^1(l^2)$.
	By the Marcinkiewicz interpolation theorem (see, for example, \cite[I.4]{St}), we obtain estimate~\eqref{GIsBounded}. 
	Passing to the limit one more time, 
	we lift the assumption about the finiteness of~$\mathcal{A}$.
\end{proof}

\begin{Le}\label{aboutR}
	Let $h$ be a function in $L^p$ or in $L^p(l^2)$\textup, $1<p\le 2$. Consider the operator~$S$ defined by the formula
	$$
		Sh \df \Big(\sum_{k=0}^{\infty} |\dltn_k h|^2\Big)^{1/2}.
	$$
	Then we have
	$$
		\|Sh\|\Cii{L^p} \le C_p \|h\|\Cii{L^p}.
	$$
\end{Le}
This estimate is well known for scalar-valued functions (moreover, in~\cite{Bu} it is proved that $\|Sh\|_{L^p} \asymp \|h\|_{L^p}$, $1 < p < \infty$).
As for our situation, Lemma~\ref{aboutR} is a simple consequence of Theorem~\ref{OperOnMart} (the arguments are the same as in the proof of Lemma~\ref{aboutS}). 

Also we will need the following simple fact.
\begin{Rem}\label{merg}
Let $\{h_i\}$ and $\{v_j\}$ be sequences in $L^p(l^2)$, $1\le p \le 2$. Then
$$
	\big\|\{h_i\}\big\|_{L^p(l^2)} + \big\|\{v_j\}\big\|_{L^p(l^2)} \le \sqrt{2}\, \big\|\{h_i\} \cup \{v_j\}\big\|_{L^p(l^2)}
$$
\end{Rem}
\begin{proof}
This fact follows from the concavity of the functions $x^{1/p}$ and $x^{p/2}$ for $x \ge 0$, i.e., we need to apply the inequality
$
	\tfrac{x^q + y^q}{2} \le \big(\tfrac{x+y}{2}\big)^{q}
$
twice: with $q= 1/p$ and $q= p/2$.
\end{proof}

\section{Proof of Theorem~\ref{th1}}
Let $I = [a,b) = [a, b - 1]$ be some interval in $\mathbb{Z}_+$. We consider the dyadic decomposition of its left end: $a = 2^{k_1} + \cdots + 2^{k_s}$, where $k_1 >k_2>\cdots> k_s \ge 0$. Also we consider the collection
$
	\{\varkappa_j\}_{j=1}^{\infty} = \mathbb{Z}_+ \setminus \{k_i\}_{i=1}^{s}
$
ordered by ascending: $\varkappa_1 < \varkappa_2 < \cdots < \varkappa_j < \cdots$.
We split the right-unbounded interval $[a, +\infty)$ into pairwise disjoint subintervals as follows:
$$
	[a, +\infty) = \{a\} \cup \bigcup_{j=1}^{\infty} J_{j},
$$
where
$$
	J_{j} \df \big[a +  {\textstyle \sum_{l=1}^{j-1} 2^{\varkappa_{l}}} + 1,\, a + {\textstyle \sum_{l=1}^{j} 2^{\varkappa_{l}}}\big].
$$
By $q$ we denote the index such that $J_{q} \cap I \ne \emptyset$ and $J_{q+1} \cap I = \emptyset$.

Next, we consider the dyadic decomposition of $b$: $b = 2^{\tilde k_1} + \cdots + 2^{\tilde k_{r}}$, 
where $\tilde k_1 > \tilde k_2>\cdots> \tilde k_{r} \ge 0$, and split the interval $[0, b - 1]$ into pairwise disjoint subintervals:
$$
	[0, b - 1] = \bigcup_{i=1}^{r} \tilde J_{i}, 
$$
where
$$
	\tilde J_{i} \df \big[{\textstyle \sum_{l=1}^{i-1} 2^{\tilde k_{l}}},\,{\textstyle \sum_{l=1}^{i} 2^{\tilde k_{l}}} - 1\big].
$$
From the collection $\{\tilde k_i\}_{i=1}^{r}$, we choose the exponent~$\tilde k_{\rho}$ such that 
$\theta_{\tilde k_{\rho}}(a) = 0$ and $\theta_k(a) = \theta_k(b)$ for $k > \tilde k_{\rho}$. Note that $\theta_{\tilde k_{\rho}}(b) = 1$.

Now we prove the identity 
\begin{equation}\label{Partition}
	I = \{a\}\cup\Big(\bigcup_{j=1}^{q - 1} J_{j}\Big) \cup \Big(\bigcup_{i = \rho+1}^{r} \tilde J_{i}\Big)
\end{equation}
as well as the fact that all the intervals in this partition
are pairwise disjoint. For this, it suffices to show that
$
	J_{q} \cap I = \big[{\textstyle \sum_{l=1}^{\rho} 2^{\tilde k_{l}}},\, b-1\big],
$
or, what is the same, that
\begin{equation}\label{EndOfLastInterval}
	a +  {\sum_{l=1}^{q-1} 2^{\varkappa_{l}}} + 1 ={\sum_{l=1}^{\rho} 2^{\tilde k_{l}}}.
\end{equation}
The number $\tilde a \df a +  {\sum_{l=1}^{q-1} 2^{\varkappa_{l}}}$ is constructed from $a$ as follows: we fill ``empty'' lower binary digits of $a$ until we get the number that is smaller than
$b$, but that will become greater if we fill one more digit. So, since $\theta_{\tilde k_{\rho}}(b) = 1$, we have
$\theta_k(\tilde a) = 1$ for $k < \tilde k_{\rho}$, $\theta_{\tilde k_{\rho}}(\tilde a) = 0$, and $\theta_k(\tilde a) = \theta_k(a) = \theta_k(b)$ for $k > \tilde k_{\rho}$. 
This implies identity~\eqref{EndOfLastInterval}. Therefore, we have proved relation~\eqref{Partition} together with the fact that all the intervals in it are pairwise disjoint.

Now we apply the procedure just described to each interval $I_m = [a_m, b_m)$.
We assign the additional index~$m$ to all the objects arising from the application of this procedure to~$I_m$. 
Also we introduce the following notation:
\begin{gather*}
	f_{m,0} \df (f_m, w_{a_m})\, w_{a_m},\\
	f_{m,j} \df \sum_{n\in J_{m,j}} (f_m, w_n)\, w_n,\qquad\mbox{and}\qquad  
	\tilde f_{m,i} \df \sum_{n\in \tilde J_{m,i}} (f_m, w_n)\, w_n. \rule{0pt}{15pt}
\end{gather*}
Since the intervals in~\eqref{Partition} are pairwise disjoint, we have
\begin{equation}\label{fm}
	f_m =  f_{m,0}+ \sum_{j = 1}^{q_m - 1} \!\! f\Cii{m,j}  + 
	\!\!\sum_{i = \rho_m + 1}^{r_m} \!\! \tilde f\Cii{m,i}.
\end{equation}
Next, we set
$g_{m,j} \df w_{a_m} f_{m,j}$ and $\tilde g_{m,i} \df w_{b_m} \tilde f_{m,i}$.
Using Facts~\ref{ZIsGroup} and \ref{WIsGroup}, we can rewrite identity~\eqref{fm} as follows:
\begin{equation*}
	f_m \,=\, w_{a_m} \Big(g_{m,0} + \sum_{j = 1}^{q_m - 1} g_{m,j}\Big)  + 
	w_{b_m} \!\! \sum_{i = \rho_m + 1}^{r_m} \!\! \tilde g_{m,i}.
\end{equation*}
Therefore, by Lemma~\ref{intervals} and Fact~\ref{MartViaWalsh} we have
$$
	f_m \,=\, w_{a_m} \Big(\dltn_0 \, g_{m,0} + \sum_{j = 1}^{q_m - 1} \dltn_{1+\varkappa_{m,j}}\,  g_{m,j}\Big)  + 
	w_{b_m} \!\! \sum_{i = \rho_m + 1}^{r_m} \!\! \dltn_{1+\tilde k_{m,i}}\, \tilde g_{m,i}.
$$
This identity, together with Lemma~\ref{aboutS}, implies that
\begin{equation*}
	\Big\|\sum_m f_m\Big\|_{L^{p}} \le C_p\,\Big\|\Big(\sum_m\sum_{j = 0}^{q_m - 1}\! |g_{m,j}|^2  + 
	\sum_m\sum_{i = \rho_m + 1}^{r_m}  \!\!|\tilde g_{m,i}|^2\Big)^{1/2} \Big\|_{L^{p}} \\
\end{equation*}
Using the triangle inequality and applying Lemma~\ref{aboutR} to one of the terms, we conclude that the last expression is not greater than
\begin{equation}\label{ExprInProof}
	C_p'\,\Big\|\Big(\sum_m\sum_{j = 0}^{q_m - 1} \!|g_{m,j}|^2\Big)^{1/2} \Big\|_{L^{p}}  + 
	C_p'\,\bigg\|\bigg(\sum_m\Big|\sum_{i = \rho_m + 1}^{r_m}\!\!  \tilde g_{m,i}\,\Big|^2\bigg)^{1/2} \bigg\|_{L^{p}}.
\end{equation}
Next, we note that
$$
	g_{m,q_m} = w_{a_m}\,w_{b_m}\!\!\sum_{i = \rho_m + 1}^{r_m}\!\!  \tilde g_{m,i}.
$$
This identity and
Fact~\ref{merg} imply that expression~\eqref{ExprInProof} can be estimated by
$$
	C_p'\sqrt{2} \,\Big\|\Big(\sum_m\sum_{j =0}^{q_m} |g_{m,j}|^2\Big)^{1/2} \Big\|_{L^{p}} = C_p'\sqrt{2} \,\Big\|\Big(\sum_m\sum_{k =0}^{\infty} |\dltn_k g_m|^2\Big)^{1/2} \Big\|_{L^{p}},
$$
where $g_m \df w_{a_m} f_m$. Applying Lemma~\ref{aboutR} once again, we see that the last expression is not greater than
\[
\pushQED{\qed} 
	C_p''\sqrt{2} \,\big\|\{g_m\}\big\|_{L^p(l^2)} = C_p''\sqrt{2} \,\big\|\{f_m\}\big\|_{L^p(l^2)}.
	\qedhere
\popQED
\]

\section{Acknowledgments}
The author expresses his gratitude to S.~V.~Kislyakov for suggesting the problem and for helpful remarks about it. The author would also like to thank A.~L.~Volberg for a fruitful discussion on the Walsh system, martingales and the Littlewood--Paley inequality.


\begin{thebibliography}{99}
	\bibitem{Ru} Jos\'e~L.~Rubio de Francia,
		\emph{A Littlewood--Paley inequality for arbitrary intervals}, Rev. Mat. Iberoamer., Vol.~1, No.~2, 1985, 1--14
	
	\bibitem{Bo} J.~Bourgain,
		\emph{On square functions on the trigonometric system}, Bull. Soc. Math. Belg., Vol.~37, No.~1, 1985, 20--26

	\bibitem{KiPa} S.~V.~Kislyakov and D.~V.~Parilov,
		\emph{On the Littlewood--Paley theorem for arbitrary intervals}, J. of Math. Sci., Vol.~139, No.~2, 2006, 6417--6424
	
	\bibitem{KaSa} B.~S.~Kashin and A.~A.~Saakyan, \emph{Orthogonal Series}, Transl. of Math. Monographs, 
		Amer. Math. Soc. (Reprint edition), Vol.~75, 2005
		
	
	\bibitem{DiUh} J.~Diestel and J.~J.~Uhl, Jr., \emph{Vector Measures}, Math. Surveys and Monographs, Amer. Math. Soc., Vol.~15, 1977
		
	\bibitem{Gu} R.~F.~Gundy, \emph{A decomposition for $L^1$-bounded martingales}, The Ann. of Math. Stat., Vol.~39, No.~1, 1968, 134--138
	
	\bibitem{Ki} S.~V.~Kislyakov, \emph{Martingale transforms and uniformly convergent orthogonal series}, 
		J. of Soviet Math., Vol.~37, No.~5, 1987, 1276--1287
	
	\bibitem{St} Elias~M.~Stein, \emph{Singular Integrals and Differentiability Properties of Functions}, Princeton University Press, 1970
	
	\bibitem{Bu} D.~L.~Burkholder,  \emph{Martingale Transforms}, The Ann. of Math. Stat., Vol.~37, No.~6, 1966, 1494--1504.
				
\end{thebibliography}
\end{document}